\documentclass[a4paper]{amsproc}
\usepackage{amssymb}
\usepackage{amscd}
\usepackage{epsfig}
\theoremstyle{plain}
 \newtheorem{thm}{Theorem}[section]
 
 \newtheorem{lem}{Lemma}[section]
 
\theoremstyle{definition}

\theoremstyle{remark}
  
 \numberwithin{equation}{section}
 \newcommand\mycom[2]{\genfrac{}{}{0pt}{}{#1}{#2}}
\renewcommand{\leq}{\leqslant}
\renewcommand{\geq}{\geqslant}

\title[Derivatives of Hardy's function]{A NEW BOUND FOR THE ERROR TERM IN THE APPROXIMATE FUNCTIONAL EQUATION FOR THE DERIVATIVES OF THE HARDY'S Z-FUNCTION}

\subjclass[2010]{11M06}

\keywords{Riemann zeta function, Hardy's function.}

\author[Blanc]{\bfseries Philippe Blanc}

\address{
D\'{e}partement des technologies industrielles \\ % \hfill (Received 00 00 2010)\\
Haute \'{E}cole d'Ing\'{e}nierie et de Gestion  \\ %\hfill (Revised  00 00 2010)\\
CH-1400 Yverdon-les-Bains\\
Switzerland}
\email{philippe.blanc@heig-vd.ch}

\begin{document}
%{\begin{flushleft}\baselineskip9pt\scriptsize
%PUBLICATIONS DE L'INSTITUT MATH\'EMATIQUE\newline
%Nouvelle s\'erie, tome 87(101) (2010), od--do \hfill DOI:
%\end{flushleft}}
\vspace{18mm}
\setcounter{page}{1}
\thispagestyle{empty}

\begin{abstract}
Lavrik and the author gave uniform bounds of the error term in the approximate functional equation for the derivatives of the Hardy's Z-function. We obtain a new bound of this error term which is much better for high order derivatives.
\end{abstract}

\maketitle

% \linenumbers

\section{Introduction and main result}
\label{}
Let $\zeta$ be the Riemann zeta function, and $Z$ the Hardy function defined by
\[
Z(t)=e^{i\theta(t)}\zeta \left(\frac{1}{2}+it\right)
\]
where
\begin{equation}\label{theta}
\theta(t)=\arg\left( \pi^{-i\frac{t}{2}}\;\Gamma\left (\frac{1}{4}+i\frac{t}{2}\right)\right)
\end{equation}
and the argument is defined by continuous variation of $t$ starting with the value $0$ at $t=0$. The real zeros of $Z$ coincide with the zeros of $\zeta $ located on the line of real part $\frac{1}{2}$. The function $\theta$ plays a central role in this paper and it is important to mention \cite{Ivic2} that 
\begin{equation}\label{thetaprop}
\theta(t) = \frac{t}{2}\log \frac{t}{2\pi} - \frac{t}{2} - \frac{\pi}{8}+
O\left (t^{-1}\right)\,
\end{equation}
and 
\begin{equation}\label{thetapropder}
\theta'(t) = \frac{1}{2}\log \frac{t}{2\pi} +
O\left (t^{-2}\right).\,
\end{equation}
A weak form of the celebrated Riemann-Siegel formula \cite{Ivic2} asserts that
\[
Z(t)=2\hspace{-0.2 cm} \sum_{1\leq n \leq \sqrt{\frac{t}{2\pi}}}\frac{1}{\sqrt{n}}\cos(\theta(t)-t\log n)+O\left (t^{-\frac{1}{4}}\right )
\]
and, concerning the derivatives of $Z$, the approximate functional equation reads
\[
Z^{(k)}(t) = 2\hspace{-2mm}\sum_{1\leq n \leq \sqrt{\frac{t}{2\pi}}}\frac{1}{\sqrt{n}}\left ( \theta'(t)-\log n\right)^{k}\cos(\theta(t)-t\log n+k\frac{\pi}{2})+R_k(t)
\]
where $R_k(t)$ is the error term. Of particular interest is the set of integers $k$, which depends on $t$,  such that, uniformly in $k$,
\begin{equation}\label{R}
R_k(t)=o(\theta'(t)^k)\,\,\mbox{as}\,\,t\to\infty
\end{equation}
which means that $R_k(t)$ is a true error term.\\
Lavrik \cite{Lav1} proved that (\ref{R}) holds for $0\leqslant k \leqslant c\hspace{0.5 mm}\theta'(t)$ where $c<\frac{1}{2\log 3}=0.4551...$, the author \cite{PBC1} extended this result to $c< 1.7955...$ and numerical experiments suggested that (\ref{R}) is probably true for larger $k$. For example $\vert R_k(10^4)\vert\leqslant 0.05\, \theta'(10^4)^k$ for $k=1,\ldots,117.$\\
In this paper, we prove that (\ref{R}) holds for $0\leqslant k \leqslant c\hspace{0.5 mm}\theta'(t)^2$ where $c<3$, which is a consequence of Theorem \ref{MainTheorem}. To simplify its proof, and since the case $0\leqslant k < \theta'(t)$ is covered by~\cite{PBC1}, we restrict our attention to the case $k\geqslant\theta'(t)$.
\begin{thm}\label{MainTheorem}
Let $t$ be large enough and $c>1$ be a fixed constant. Then, for $\theta'(t)\leqslant k\leqslant 3\hspace{0.2 mm}\theta'(t)^2$, we have, uniformly in $k$,
\[
R_k(t)=O\left(t^{-\frac{1}{2}}\,c^{\frac{k}{\theta'(t)}}\theta'(t)^k+t^{-\frac{3}{4}}e^{\frac{k}{2\theta'(t)}}k\hspace{0.5 mm}\theta'(t)^{k-1}\right).
\]
\end{thm}
 The notations used in this paper are standard : $\lfloor x\rfloor$ and $\lceil x\rceil$ stand for the usual floor and ceiling functions and $\{x\}:=x-\lfloor x\rfloor$.
 %The writing $f(x)\ll g(x)$ means there exist $x_0$ and a positive constant $C$ such that $\left |f(x)\right |\leq Cg(x)$ for $x\geq x_0$ and $f(x)\asymp g(x)$ means that both $f(x)\ll g(x)$ and $g(x)\ll f(x)$ hold.
  We denote by $(x)_n$ the Pochhammer symbol defined by $(x)_0=1$ and $(x)_n=x(x+1)\cdots(x+n-1)$ for $n\in\mathbb{N}^*$.\\
 In the next section, we recall some results of the author used in the proof of Theorem~1 of \cite{PBC1} and we state the lemmas needed in the proof of our main result. Section 3 is devoted to the proofs. 
 \section{Preliminary results}
The functions $\eta_p(d,s):=\sum_{n=1}^{\infty}n^{-s}(d-\log n)^p$ defined for $d\in\mathbb{R}$, $p\in\mathbb{N}$ and~$\Re(s)>~1$ have a meromorphic extension to $\Re(s)>0$ with a pole at $s=1$ and, as a consequence of the Fa\`a di Bruno formula \cite{Goursat}, we have
\begin{equation}\label{Zk}
  Z^{(k)}(t)=e^{i\theta(t)}i^k\eta_k(\theta'(t),\frac{1}{2}+it)+e^{i\theta(t)}\sum_{p=0}^{k-2} q_p(t)i^p\eta_p(\theta'(t),\frac{1}{2}+it) 
  \end{equation}
  where
    \begin{equation}\label{qp}
      q_p(t)\,=\hspace{-0.4 cm} \sum_{\mycom{2p_2+3p_3+\ldots+k p_k=k-p}{ p_2\geq 0,\, p_3\geq 0,\ldots,\, p_k\geq 0}}\frac{k!}{p!p_2!\,\cdots\, p_k!}\left(\frac{i\,\theta''(t)}{2!}\right)^{p_2}
       \! \left(\frac{i\,\theta'''(t)}{3!}\right)^{p_3}\!\cdots\left(\frac{i\,\theta^{(k)}(t)}{k!}\right)^{p_k}\!\!\!\!.
     \end{equation}
The first step in our proof is to get an approximate functional equation for the functions $\eta_p(\theta'(t),\frac{1}{2}+it)$. In \cite{PBC1}, we proved that the functions $\tilde{\eta}_p(d,s):=(-1)^p\eta_p(d,s)$ satisfy, for $d=\theta'(t)$ and $s=\frac{1}{2}+it$, the relation 
    \begin{eqnarray} \label{B}
    \tilde{\eta}_p(d,s)&=&\sum_{1 \leq n \leq N}\phi_p(n)+\frac{N^{1-s}}{s-1}(\log N -d)^p\, p!\,\sum_{l=0}^{p}\frac{((s-1)(\log N -d ))^{-l}}{(p-l)!}\nonumber\\
    &+&O\left( (1+ \vert t \vert )N^{-\frac{1}{2}}\log^p N \right)
   \end{eqnarray}
   where $\phi_p(x):=x^{-s}(\log x -d)^p$.\\
In this paper, we fix a constant $c>1$ and for $t$ sufficiently large, we set  $N_0=\lceil e^{d}\rceil  $, $N_1=\lfloor c\hspace{0.2mm}e^{2d}\rfloor$ and for $N>N_1$ we split the sum
    \begin{equation}\label{Sum}
    \sum_{1 \leq n \leq N}\phi_p(n)=\sum_{1 \leq n \leq N_0}\phi_p(n)+\sum_{N_0< n\leq N_1 }\phi_p(n)+\sum_{N_1 < n \leq N}\phi_p(n).
   \end{equation}
We use Lemmas \ref{lemquinties} \cite{PBC2,PBC3} and \ref{lemquaterbis}  to transform the second sum in a short sum (Lemma \ref{appeta}), and Lemma \ref{MacLaurin} to apply the Euler-MacLaurin formula to the third sum (Lemma \ref{thirdsum}).
\begin{lem}\label{lemquinties}
 Let $a < b$ be integers and let $\varphi\in C^2[a,b]$ and $f\in C^5[a,b]$ be real-valued functions possessing the following property: There exist constants $H>0$, $U\geq b-a$ and $1\ll A \ll U$ satisfying
  \[
 f''(x)\asymp  A^{-1},\hspace{2mm} f'''(x)\ll A^{-1}U^{-1},\hspace{2mm}f^{(4)}(x)\ll A^{-1}U^{-2},\hspace{2mm} f^{(5)}(x)\ll A^{-1}U^{-3}
 \]
 \[
  \varphi (x)\ll H,\hspace{2mm} \varphi'(x)\ll H U^{-1},\hspace{2mm} \varphi''(x)\ll H U^{-2}
 \]
 for all $x\in [a,b]$.\\
 Let furthermore $\Theta$ be the function defined on $ ]0,\infty[\times [0,1]$ by
 \[
 \Theta(\lambda,\mu)=i\int\limits_0^{\infty}\frac{\sinh(2\pi(\mu-\frac{1}{2})x)}{\sinh (\pi x)}e^{-i\pi\lambda x^2}\,dx
 \]
 and $x(\cdot)$ be the unique function defined by $f'(x(y))=y$ for all $y\hspace{-0.7mm}\in\hspace{-0.7mm}[f'(a),\hspace{-0.05mm}f'(b)]$. Then
 \begin{eqnarray}\label{transform}
 \sum_{a<n\leq b}\varphi(n)e^{2\pi i f(n)}&=&e^{i\frac{\pi}{4}}\sum_{f'(a)<n\leq f'(b)}\frac{\varphi(x(n))}{\sqrt{f''(x(n))}}e^{2\pi i(f(x(n))-nx(n))}\\
 &+&R(b)-R(a)+O(H)\nonumber
  \end{eqnarray}
  where
  \[
  R(l)=\varphi(l)e^{ 2\pi i f(l)}\Theta(f''(l),\{f'(l)\}).
  \]
  \end{lem}

 \begin{lem}\label{lemquaterbis}
 For  $d\in\mathbb{R}$ and $p\in \mathbb{N}$, let $\varphi_p$ be the function defined by $ \varphi_p(x)=x^{-\frac{1}{2}}(\log x - d)^p$ for $x\geqslant 1$. Then, for $2\leqslant d\leqslant p$, $e^d\leqslant a< c\hspace{0.2mm}e^{2d}$ and $x\in[a,(2a)^*]$ where $(2a)^*= \min(2a,c\hspace{0.2mm}e^{2d})$, we have
 \[
  \varphi_p(x)\ll H,\hspace{1mm}\varphi_p'(x)\ll Ha^{-1},\hspace{1mm} \varphi_p''(x)\ll Ha^{-2}
 \]
 where $H=p^2\varphi_{p-2}((2a)^*)$.
 \end{lem}

 \begin{lem}\label{appeta}
 Let $t$ be large enough and assume that $\theta'(t)\leqslant p\ll t^{\frac{1}{2}}$. Then
 \begin{equation}\label{secondsum}
\sum_{N_0 < n\leq N_1 }\!\!\!\!\!\!\phi_p(n)\!=\!\,e^{-2 i\theta(t)}\sum_{1 \leqslant n \leqslant N_0}\frac{(\theta'(t)-\log n )^p}{n^{\frac{1}{2}-it}}+O\left(t^{-\frac{1}{2}}p^2\,c^{\frac{p}{\theta'(t)}}\theta'(t)^{p-2}\right).
 \end{equation}
 \end{lem}
 
The next lemma prepares the application of the Euler-MacLaurin formula to the third sum of (\ref{Sum}). In \cite{PBC1}, the bound we got for the third sum, which depends on an upper bound for $\vert \phi_p \vert$ on $[N_1,N]$, is not optimal for $p\geqslant d$.
 
 \begin{lem}\label{MacLaurin}
 For $d\in\mathbb{R}$, $p\in \mathbb{N}$ and $s=\frac{1}{2}+i t$, let $g_p$ and $\phi_p$  be the function defined by $g_p(x)=(\log x -d)^p$ and $\phi_p(x)=x^{-s}g_p(x)$ for $x\geqslant 1$. Then, for $0< d\leqslant p$, $x\in[c\hspace{0.2 mm}e^{2d},\infty[$ and $k\in\mathbb{N}$, we have
 \begin{equation}\label{derg}
 \left| g_p^{(k)}(x)\right|\leqslant k!\left(\frac{p}{d}\right)^k x^{-k}g_p(x).
 \end{equation}
Further, for $t$ large enough, let $d=\theta'(t)$, $K \asymp t^{\frac{1}{2}}$, $N_1=\lfloor ce^{2d}\rfloor $, $N_2=N_1+1$ and let $N>N_2$. Then, for $d\leqslant p\ll t^{\frac{1}{2}}$ and   $0\leqslant k\leqslant 2K$, we have
\[
 \phi_p^{(k)}(x)=(-1)^k(s)_k\,x^{-s-k}g_p(x)\left(1+O\left(d^{-1}\right)\right)
\]
for $x\in[c\hspace{0.2mm}e^{2d},\infty[$
and moreover
\begin{equation}\label{derphi}
\phi_p^{(k)}(N_2)\ll \left(\frac{2\pi}{c}\right)^k t^{-\frac{1}{2}}c^{\frac{p}{d}}d^p,
\end{equation}
\begin{equation}\label{derphiN}
\phi_p^{(k)}(N)\ll \left(\frac{2\pi}{c}\right)^k N^{-\frac{1}{2}}\log^p N
\end{equation}
and
\begin{equation}\label{intphi}
  \int_{N_2}^N\vert \phi_p^{(2K)}(u)\vert\,du \ll  \left(\frac{2\pi}{c}\right)^{2K}c^{\frac{p}{d}}d^p+\left(\frac{t}{N}\right)^{2K-\frac{1}{2}}\log^p N.
\end{equation}
 \end{lem}
\begin{lem}\label{thirdsum}
 Let $t$ be large enough and assume that $\theta'(t)\leqslant p\ll t^{\frac{1}{2}}$. Then
\begin{eqnarray}\label{thirdsumbis}
\sum_{N_1 < n \leq N}\phi_p(n)&=&-\frac{N^{1-s}}{s-1}(\log N -d)^p\, p!\,\sum_{l=0}^{p}\frac{((s-1)(\log N -d ))^{-l}}{(p-l)!}\\
&+&O\left(t^{-\frac{1}{2}}c^{\frac{p}{d}}d^p+N^{-\frac{1}{2}}\log^p N\right).\nonumber
\end{eqnarray}
\end{lem}
Finally, the next lemmas are needed to make use of relation (\ref{Zk}).
\begin{lem}\label{etaapp}
Let $t$ be large enough, $c>1$ be a fixed constant and assume that $\theta'(t)\leqslant p\ll t^{\frac{1}{2}}$.Then
\begin{eqnarray}\label{etafunc}
\eta_p\left(\theta'(t),\frac{1}{2}+it\right)&=&\sum_{1 \leq n \leq \sqrt{\frac{t}{2\pi}}}\frac{(\theta'(t)-\log n)^p }{n^{\frac{1}{2}+it}}+ e^{- 2i\theta(t)}\sum_{1 \leq n \leq \sqrt{\frac{t}{2\pi}}}\frac{(\log n-\theta'(t))^p }{n^{\frac{1}{2}-it}}\nonumber
\\&+& O\left(t^{-\frac{1}{2}}c^{\frac{p}{\theta'(t)}}\theta'(t)^p\right).
 \end{eqnarray}
\end{lem}
  \begin{lem}\label{lemquinquiesbis}
 Let $\theta$ be the function defined by (\ref{theta}). Then, for $\nu\geq 2$ and $t>0$ we have
  \[
 \left |\theta^{(\nu)}(t)\right|\leq \frac{(\nu-2)!}{2t^{\nu-1}}+\frac{2\nu !}{\sqrt{\nu}\,t^{\nu}}\,.
 \]
 Further, let $t$ be large enough and assume that $ k\ll t^{\frac{1}{2}}$. Then
 \[
  \sum_{\nu=1}^k\frac{\left|\theta^{(\nu)}(t)\right|t^{\nu}}{\nu!}\leq t\theta'(t)+\frac{t}{2}
 \hspace{1cm}\mbox{and}\hspace{1 cm} \sum_{p=0}^{k-2} \vert q_{p}(t)\vert\theta'(t)^{p}\ll \frac{k}{t}e^{\frac{k}{2\theta'(t)}}\theta'(t)^{k-1}.
     \]
  where $q_p$ are the functions defined by (\ref{qp}).
 \end{lem}

\section{Proofs}
\begin{proof}[ Proof of Lemma~{\rm\ref{lemquaterbis}}]
 
 By computing $\varphi_p'$ and $\varphi_p''$ we see that
 \[
\varphi_p'(x) \ll \max(\varphi_p(x),p\hspace{0.5mm}\varphi_{p-1}(x))a^{-1}
 \]
 and
 \[
\varphi_p''(x) \ll \max(p\hspace{0.5mm}\varphi_{p-1}(x),\varphi_p(x)+p^2\varphi_{p-2}(x))a^{-2}
 \]
for $x\geqslant a$ and we complete the proof by noting that for $p\geqslant d$ and $x\in[e^d,c\hspace{0.2mm}e^{2d}]$ we have
 \[
 \varphi_p(x)\ll p\varphi_{p-1}(x)\ll p^2 \varphi_{p-2}(x)
 \]
and that
 \[
 \varphi_{p-2}(x) \leqslant  a^{-\frac{1}{2}}(\log (2a)^* -d)^{p-2}\leqslant 2^{\frac{1}{2}}\varphi_{p-2}((2a)^*)
 \]
for $x\in[a,(2a)^*]$.
 \end{proof}
  \begin{proof}[ Proof of Lemma~{\rm\ref{appeta}}]
      We have
     \begin{equation}\label{231}
     \sum_{N_0 < n\leq N_1 }\phi_p(n)=\sum_{0\leq r \leq l}\;\sum_{a_r < n \leq a_{r+1}}\;\phi_p(n)
     \end{equation}
     where $l$ is an integer such that $2^l N_0 < N_1 \leq 2^{l+1}N_0$, $a_r=2^r N_0$ for $r=0,\ldots,l$ and $a_{l+1}=N_1$. We introduce the functions $\varphi_p(x)=x^{-\frac{1}{2}}(\log x -d)^p$ and $f(x)= -\frac{t}{2\pi}\log x $ so that $\phi_p(n)=\varphi_p(n)e^{2\pi i f(n)}$. Since $e^{d}\leqslant N_0<N_1\leqslant c\hspace{0.2mm}e^{2d}$ where $d=\theta'(t)$ and thanks to Lemma \ref{lemquaterbis}, the assumptions of Lemma \ref{lemquinties} are satisfied with
    $a=a_r$, $b=a_{r+1}$, $H=H_r:=p^2\varphi_{p-2}(a_{r+1})$, $U=a_r$, $A=\frac{a_r^2}{t}$ and relation (\ref{transform}) reads
     \begin{eqnarray*}
     \sum_{a_r < n \leq a_{r+1}}\;\phi_p(n)&=&e^{i\frac{\pi}{4}}\sum_{f'(a_r)<n\leq f'(a_{r+1})}\frac{\varphi_p(x(n))}{\sqrt{f''(x(n))}}e^{2\pi i (f(x(n))-nx(n))}\nonumber\\
    &+&R(a_{r+1})-R(a_r)+O(H_r)
    \end{eqnarray*}
    where $x(n)=-\frac{t}{2\pi n}$. By definition $\Theta(\lambda,\mu)=O(\lambda^{-\frac{1}{2}})$ for $\mu\in[0,1]$ and  $\Theta(\lambda,\mu)=O_{\delta}(1)$ for $\mu\in[\delta,1-\delta]$ which imply that $R(a_0)$ is a $O(H_0)$ and $R(a_{l+1})$ is a~$O(H_l)$. We sum the previous relations, noting that $\lfloor f'(a_{l+1})\rfloor =-1$ and setting $q=-f'(a_0)$, to get
     \begin{eqnarray}\label{232}
    \sum_{0\leq r \leq l}\;\sum_{a_r < n \leq a_{r+1}}\phi_p(n)&=&\hspace{-4mm}\sum_{-q<n\leq-1}\frac{\varphi_p(x(n))}{\sqrt{f''(x(n))}}e^{ 2\pi i (f(x(n))-nx(n)+ \frac{1}{8})}+O(\sum_{0\leq r \leq l}H_r)\nonumber\\
     &=&\hspace{-4mm}\sum_{1 \leq n <q}\hspace{-1mm}\frac{\varphi_p(x(-n))}{\sqrt{f''(x(-n))}}e^{ 2\pi i (f(x(-n))+nx(-n)+\frac{1}{8})}\!+\!O(\!\sum_{0\leq r \leq l}H_r).
    \end{eqnarray}
    Further, using (\ref{thetapropder}) and $d=\theta'(t)$, we have for $1\leq n\leq q$
    \begin{equation}\label{233}
    \frac{\varphi_p(x(-n))}{\sqrt{f''(x(-n))}}=\frac{(\log \frac{t}{2\pi n}-d)^p}{n^{\frac{1}{2}}}=\frac{(\theta'(t)-\log n)^p}{n^{\frac{1}{2}}}+O(pt^{-2}\theta'(t)^{p-1})
     \end{equation}
    and making use of (\ref{thetaprop}) we get
    \begin{equation}\label{234}
    e^{  2\pi i (f(x(-n))+nx(-n)+\frac{1}{8})}=e^{-2 i ( \frac{t}{2}\log\frac{t}{2\pi}-\frac{t}{2}-\frac{\pi}{8})+it\log n}=\frac{e^{ -2i\theta(t)}}{n^{-it}}+O(t^{-1}).
     \end{equation}
     Further, for $m\geqslant d-2$, the function $\varphi_m$ is increasing on $[N_0,N_1]$ and we have
      \begin{equation}\label{var}
          \sum_{r=0}^{l}\varphi_m(a_{r+1})\leqslant \sum_{r=1}^{l-1}\varphi_m(2^rN_0)+2\varphi_m(N_1)
             \leqslant \int_1^{l}\varphi_m(2^uN_0)\,du +2\varphi_m(N_1)
         \end{equation}
     and since 
     \[
    d\leqslant \log N_0\leqslant \log (e^d+1)=d+\log(1+e^{-d})\leqslant d+e^{-d}
     \]
  and $e^{-d}=O(t^{-\frac{1}{2}})$, we get for $u\geqslant 1$ and $m\ll t^{\frac{1}{2}}$
     \[
     \left(u\log 2+\log N_0-d\right)^m\!\leqslant\! (u\log 2)^m\left(1+\frac{e^{-d}}{u\log 2}\right)^m\!\!\!\leqslant (u\log 2)^m e^{\frac{me^{-d}}{u\log 2}}\ll (u\log 2)^m
     \]
     and therefore
     \begin{equation}\label{varphi}
     \int_1^{l}\varphi_m(2^uN_0)\,du \ll e^{-\frac{d}{2}}\int_0^l 2^{-\frac{u}{2}}\left(u\log 2\right)^m\,du\ll e^{-\frac{d}{2}}2^m\int_0^\frac{l\log 2}{2}e^{-y}y^m dy.
     \end{equation}
    Moreover $2^lN_0\leqslant N_1$ and thus $2^l\leqslant\frac{N_1}{N_0}\leqslant c\hspace{0.2mm}e^d$ and $l\log 2\leqslant d+\log c$ and this implies that
    \begin{equation}\label{intvar}
    \int_0^\frac{l\log 2}{2}e^{-y}y^m dy\leqslant \int_0^{\frac{d}{2}+\frac{\log c}{2}}e^{-y}y^m dy=\gamma(m+1,\frac{d}{2}+\frac{\log c}{2})
    \end{equation}
    where $\gamma(n,x)=\int_0^xe^{-t}t^{n-1}\,dt$ is the lower incomplete gamma function. Setting $x=\frac{d}{2}+\frac{\log c}{2}$ and using integration by parts one checks that
    \begin{equation}\label{gamma}
    \gamma(m+1,x)=m!\,e^{-x}\sum_{n=m+1}^{\infty}\frac{x^n}{n!}=e^{-x}\frac{x^{m+1}}{m+1}\sum_{k=0}^{\infty}\frac{x^k}{(m+2)_k}\ll e^{-\frac{d}{2}}\frac{2^{-m}c^{\frac{m}{d}}d^{m+1}}{m}.
    \end{equation}
 Using relations (\ref{var}), (\ref{varphi}), (\ref{intvar}), (\ref{gamma}) with $m=p-2$ we get
    \begin{equation}\label{235}
    \sum_{0\leq r \leq l}H_r\ll e^{-d}p\,c^{\frac{p}{d}}\,d^{p-1}+p^2\varphi_{p-2}(N_1)\ll e^{-d}p^2c^{\frac{p}{d}}d^{p-2}
    \end{equation}
   since $\varphi_{p-2}(N_1)\ll e^{-d}c^{\frac{p}{d}}d^{p-2}$. To complete the proof we make use of relations (\ref{231}), (\ref{232}), (\ref{233}), (\ref{234}), (\ref{235}) and we observe that the sum over $n$ such that $1\leq n<q$ can be replaced by the sum over $n$ such that $1\leq n\leqslant N_0$ without changing the order of the error term.
  \end{proof}
    \begin{proof}[ Proof of Lemma~{\rm\ref{MacLaurin}}]
   One can check by induction that the derivatives of $g_p$ are given by
    \[
    g_p ^{(k)}(x)=x^{-k}\sum_{l=0}^k c_{k\,,\,l}(p-l+1)_l(\log x-d)^{p-l} 
    \]
   where the $c_{k\,,\,l}$ are integers defined recursively by
    \[
    \left\lbrace
    \begin{array}{l}
    c_{\,0\,,\,0}=1, \;\;c_{k\,,\,0}\,=\,c_{\,0\,,\,l}=0\;\hspace{2.2 mm} \mbox{for}\;k,l\geq 1\, \\
    c_{k+1\,,\,l}\,=\,c_{k\,,\,l-1}\,-\,k c_{k\,,\,l}\;\hspace{6 mm}\mbox{for }k\geq 0,\,l\geq 1\,. 
    \end{array}
    \right.
    \]
    This shows that $c_{k\,,\,l}=S_k^l$ where the $S_k^l$ are the Stirling numbers of first kind. Hence
    \begin{eqnarray*}
     \left| g_p^{(k)}(x)\right|&\leqslant & x^{-k}(\log x-d)^p\sum_{l=0}^k \left|S_k^l\right|(p-l+1)_l(\log x-d)^{-l}\\
                               &\leqslant & x^{-k}(\log x-d)^p\sum_{l=0}^k \left|S_k^l\right|\left(\frac{p}{d}\right)^l
    \end{eqnarray*}
    since $(\log x -d)^{-l}\leqslant d^{-l}$ for $x\geqslant c\hspace{0.2 mm}e^{2d}$. Setting $y=\frac{p}{d}\geqslant 1$, we complete the proof of (\ref{derg}) by noting that
    \[
    \sum_{l=0}^k \left|S_k^l\right|y^l=(y)_k=
  (1+\frac{1}{y})(1+\frac{2}{y})\ldots(1+\frac{k-1}{y})y^k\leqslant k!\,y^k.
    \]
      By the general Leibniz rule we have
    \begin{eqnarray*}
   \phi_p^{(k)}(x)&=&(x^{-s})^{(k)}g_p(x)+\sum_{l=1}^k\binom{k}{l}(x^{-s})^{(k-l)}g_p^{(l)}(x)\\
                  &=&(-1)^k(s)_kx^{-s-k}g_p(x)(1+R)
    \end{eqnarray*}
    where
    \[
    \vert R\vert \leqslant \sum_{l=1}^k\frac{k^l}{\vert(s+k-l)_l\vert}\left(\frac{p}{d}\right)^l\leqslant \sum_{l=1}^k\left(\frac{kp}{td}\right)^l\ll d^{-1}.
    \]
  Since $N_2=c\,e^{2d}\left(1+O(t^{-1})\right)=c\frac{t}{2\pi}\left(1+O(t^{-1})\right)$ and $k\ll t^{\frac{1}{2}}$ we deduce that
  \begin{eqnarray*}
  \phi_p^{(k)}(N_2)&\ll& t^k\left| \left(\frac{1}{2t}+i\right)\left(\frac{3}{2t}+i\right)\ldots\left(\frac{2k-1}{2t}+i\right)\right|\left(c\,\frac{t}{2\pi}\right)^{-\frac{1}{2}-k}c^{\frac{p}{d}}d^p\\
  &\ll&\left(\frac{2\pi}{c}\right)^k t^{-\frac{1}{2}}c^{\frac{p}{d}}d^p
  \end{eqnarray*}
  and similarly
  \[
  \phi_p^{(k)}(N)\ll \left(\frac{t}{N}\right)^k N^{-\frac{1}{2}}\log^p N\ll \left(\frac{2\pi}{c}\right)^kN^{-\frac{1}{2}}\log^p N.
  \]
  Finally
  \[
  \int_{N_2}^N\vert \phi_p^{(2K)}(u)\vert\,du \ll\vert (s)_{2K}\vert \int_{N_2}^N u^{-\frac{1}{2}-2K}(\log u-d)^p\,du 
  \]
  \[
=\vert (s)_{2K}\vert \frac{u^{\frac{1}{2}-2K}}{\frac{1}{2}-2K}(\log u -d)^p\,\sum_{l=0}^{p}\left(\frac{p}{(2K-\frac{1}{2})(\log u -d )}\right)^l\left(\frac{p!}{p^l(p-l)!}\right)\Big{\vert}_{N_2}^N 
\]
\[
\ll \left(\frac{2\pi}{c}\right)^{2K}c^{\frac{p}{d}}d^p+\left(\frac{t}{N}\right)^{2K-\frac{1}{2}}\log^p N.
\]
    \end{proof}
     \begin{proof}[ Proof of Lemma~{\rm \ref{thirdsum}}]
    We set $N_2=N_1+1$ and we use the Euler-MacLaurin formula with $K\asymp t^{\frac{1}{2}}$ to get
     \begin{eqnarray*}
     \sum_{N_1< n \leq N}\phi_p(n)&=&\int_{N_2}^N\phi_p(u)\,du+\frac{1}{2}(\phi_p(N_2)+\phi_p(N))\\
     &+&\sum_{l=1}^K\frac{B_{2l}}{(2l)!}(\phi_p^{(2l-1)}(N)-\phi_p^{(2l-1)}(N_2)) +R_{2K}
     \end{eqnarray*}
     where
     \[
    \vert R_{2K}\vert\leqslant\frac{2\zeta(2K)}{(2\pi)^{2K}}\int_{N_2}^N\vert \phi_p^{(2K)}(u)\vert\,du.
     \]
     We have 
     \[
     \int_{N_2}^N\phi_p(u)\,du=-\frac{u^{1-s}}{s-1}(\log u -d)^p\, p!\,\sum_{l=0}^{p}\frac{((s-1)(\log u -d ))^{-l}}{(p-l)!}\Big{\vert}_{N_2}^N\\
     \]
     and we observe that
     \begin{eqnarray*}\label{C1}
     & &\frac{N_2^{1-s}}{s-1}(\log N_2 -d)^p\, p!\,\sum_{l=0}^{p}\frac{((s-1)(\log N_2 -d ))^{-l}}{(p-l)!}\\
      &=&\frac{N_2^{1-s}}{s-1}(\log N_2 -d)^p\sum_{l=0}^{p}\left(\frac{p}{(s-1)(\log N_2 -d )}\right)^l\left(\frac{p!}{p^l(p-l)!}\right)\\
      &=&\frac{N_2^{1-s}}{s-1}(\log N_2 -d)^p\left(1+O\left(t^{-\frac{1}{2}}\right)\right)\ll t^{-\frac{1}{2}}c^\frac{p}{d}d^p.
      \end{eqnarray*}
      Further, $\phi_p(N_2)\ll t^{-\frac{1}{2}}c^{\frac{p}{d}}d^p$, $\phi_p(N)\ll N^{-\frac{1}{2}}\log ^p N$ and since $B_{2l}=(-1)^{l-1}\frac{2(2l)!}{(2\pi)^{2l}}\zeta(2l)$ we have thanks to (\ref{derphi}) and (\ref{derphiN})
      \begin{eqnarray*}
      \sum_{l=1}^K\frac{B_{2l}}{(2l)!}(\phi_p^{(2l-1)}(N)-\phi_p^{(2l-1)}(N_2 ))&\ll& \sum_{l=1}^K\frac{1}{c^{2l-1}}\left(t^{-\frac{1}{2}}c^{\frac{p}{d}}d^p+N^{-\frac{1}{2}}\log^p N\right)\\
      &\ll&t^{-\frac{1}{2}}c^{\frac{p}{d}}d^p+N^{-\frac{1}{2}}\log^p N
      \end{eqnarray*}
      since $c>1$ is a fixed constant. Finally, since $K\asymp t^{\frac{1}{2}}$, we have $c^{-2K}\ll t^{-\frac{1}{2}}$ and we deduce from (\ref{intphi}) that  $R_{2K}\ll t^{-\frac{1}{2}}c^{\frac{p}{d}}d^p+N^{-\frac{1}{2}}\log^p N$.
     \end{proof}
     \begin{proof}[ Proof of Lemma~{\rm \ref{etaapp}}]
     We use the relations (\ref{B}) and (\ref{Sum}), (\ref{secondsum}) and (\ref{thirdsumbis}) with $c$ replaced by $c^{\frac{1}{2}}$, and we let $N$ tend to infinity to get
    \begin{eqnarray*}
   \tilde{\eta}_p\left(\theta'(t),\frac{1}{2}+it\right)&=&\sum_{1 \leq n \leq N_0}\frac{(\log n-\theta'(t))^p }{n^{\frac{1}{2}+it}}+ e^{- 2i\theta(t)}\sum_{1 \leq n \leq N_0}\frac{(\theta'(t)-\log n)^p }{n^{\frac{1}{2}-it}}\nonumber
    \\&+& O\left(t^{-\frac{1}{2}}p^2c^{\frac{p}{2\theta'(t)}}\theta'(t)^{p-2}\right).
     \end{eqnarray*}
     We note that $N_0$ can be replaced by $\sqrt{\frac{t}{2\pi}}$ without changing the order of the error term and to complete the proof, we use the relation $\eta(d,s)=(-1)^p\tilde{\eta}_p(d,s)$ and the inequality $x^2c^{\frac{x}{2}}\leqslant C_1c^x$ which holds for $x\geqslant 0$ and $C_1=16e^{-2}(\log c)^{-2}$ to check that $p^2c^{\frac{p}{2\theta'(t)}}\theta'(t)^{p-2}\leqslant C_1c^{\frac{p}{\theta'(t)}}\theta'(t)^p$.
     \end{proof}
     \begin{proof}[Proof of Lemma~{\rm \ref{lemquinquiesbis}}]
     The proof of this lemma is almost exactly the same as that of Lemma 7 of \cite{PBC1}. The only modification is at the end of the proof and it reads 
     \[
      \hspace{1mm}\sum_{p=0}^{k-2}\vert q_{p}(t)\vert\theta'(t)^{p}\hspace{1mm}\leq
      \]
     \begin{eqnarray*}
     & &\frac{k!}{t^k}\sum_{m=1}^{k-1}\frac{1}{m!}\left(\sum_{\nu=1}^k\frac{\left|\theta^{(\nu)}(t)\right|t^{\nu}}{\nu!}\right)^m
      \leq \frac{k!}{t^k}\sum_{m=1}^{k-1}\frac{1}{m!}\left(  t\theta'(t)+\frac{t}{2}\right)^m \leqslant\\
       & & \frac{k!}{t^k}\frac{1}{(k-1)!}\left(  t\theta'(t)+\frac{t}{2}\right)^{k-1}\left(\sum_{l=0}^{k-2}\left(\frac{k}{ t\theta'(t)+\frac{t}{2}}\right)^l\right)\ll\\
      & & \frac{k}{t}\theta'(t)^{k-1}\left(1+\frac{1}{2\theta'(t)}\right)^{k-1}\leqslant\frac{k}{t}e^{\frac{k}{2\theta'(t)}}\theta'(t)^{k-1}.
     \end{eqnarray*}
     \end{proof}
     \begin{proof}[ Proof of Theorem~{\rm \ref{MainTheorem}}]
     % Without restriction, we can assume that $1<c\leqslant~e^{\frac{1}{2}}$. 
      Thanks to (\ref{etafunc}), the first term of the right hand side of (\ref{Zk}) reads
     \[
     e^{i\theta(t)}i^k\eta_k(\theta'(t),\frac{1}{2}+it)=
     \]
      \[
      \hspace{-2mm} \sum_{1\leq n \leq \sqrt{\frac{t}{2\pi}}}\hspace{-2mm}\frac{(\theta'(t)-\log n )^k}{n^{\frac{1}{2}}}\left( e^{ i(\theta(t)-t\log n+k\frac{\pi}{2})}+
     e^{ -i(\theta(t)-t\log n+k\frac{\pi}{2})}\right)+O\left(t^{-\frac{1}{2}}c^{\frac{k}{\theta'(t)}}\theta'(t)^{k}\right)=
     \]
      \[
     2\hspace{-2mm}\sum_{1\leq n \leq \sqrt{\frac{t}{2\pi}}}\frac{1}{\sqrt{n}}\left ( \theta'(t)-\log n\right)^{k}\cos(\theta(t)-t\log n+k\frac{\pi}{2})
      +O\left(t^{-\frac{1}{2}}c^{\frac{k}{\theta'(t)}}\theta'(t)^{k}\right).
     \]
     For $\theta'(t)\leqslant p\leqslant 3\theta'(t)^2$, a trivial estimate of the right hand side of (\ref{etafunc}), with the choice $c=e^{\frac{1}{2}}$, leads to 
     \[
     \eta_p(\theta'(t),\frac{1}{2}+it)=O\left(t^{\frac{1}{4}}\theta'(t)^p\right)+O\left(t^{-\frac{1}{2}}e^{\frac{p}{2\theta'(t)}}\theta'(t)^{p}\right)=O\left(t^{\frac{1}{4}}\theta'(t)^p\right)
     \]
     since $ t^{-\frac{1}{2}}e^{\frac{p}{2\theta'(t)}}\leqslant t^{-\frac{1}{2}}e^{\frac{3}{2}\theta'(t)}\leqslant t^{\frac{1}{4}}$. In \cite{PBC1}, we proved that the same bound holds for $0\leqslant p\leqslant \theta'(t)$ and Lemma \ref{lemquinquiesbis} implies that the second term of the right hand side of (\ref{Zk}) satisfies
     \[
     e^{i\theta(t)}\sum_{p=0}^{k-2} q_p(t)i^p\eta_p(\theta'(t),\frac{1}{2}+it)=O\left(t^{-\frac{3}{4}}e^{\frac{k}{2\theta'(t)}}k\hspace{0.5 mm}\theta'(t)^{k-1}\right).
    \]
     \end{proof}

\end{document}